\documentclass[12pt]{amsart}

\usepackage{amsmath}
\usepackage{amsfonts}
\usepackage{amsthm}

\usepackage[cmtip,arrow]{xy}
\usepackage{pb-diagram, pb-xy}
\usepackage{pb-diagram, pb-xy}

\newtheorem{theorem}{Theorem}
\newtheorem{corollary}{Corollary}
\newtheorem{lemma}{Lemma}
\newtheorem{proposition}{Proposition}
\newtheorem{definition}{Definition}

\begin{document}
\title{The class of Gorenstein injective modules is covering if and only if it is closed under direct limits}
\author{Alina Iacob}

%

\maketitle %

\begin{abstract}
We prove that the class of Gorenstein injective modules, $\mathcal{GI}$, is special precovering if and only if it is covering if and only if it is closed under direct limits. This adds to the list of examples that support Enochs' conjecture:\\ "Every covering class of modules is closed under direct limits".\\
We also give a characterization of the rings for which $\mathcal{GI}$ is covering: the class of Gorenstein injective left $R$-modules is covering if and only if $R$ is left noetherian, and such that character modules of Gorenstein injective left $R$ modules are Gorenstein flat.
\end{abstract}

\section{introduction}
Precovers and preenvelopes are fundamental concepts in relative homological
algebra and they are important in many areas of mathematics. The importance of precovers comes from the fact that their existence allows constructing resolutions with respect to a class of modules $\mathcal{C}$.
The existence of $\mathcal{C}$-covers allows constructing minimal such resolutions (which are unique up to isomorphisms).\\

We are interested here in Gorenstein injective precovers and covers. The existence of the Gorenstein injective envelopes over arbitrary rings was recently proved in \cite{saroch.stovicek}. But the question "Over which rings is the class of Gorenstein injective modules (pre)covering?" is still open. It is known that the existence of the Gorenstein injective covers implies that the ring is noetherian (\cite{christensen:11:beyond}).  We prove that the class of Gorenstein injective left $R$-modules is covering if and only if it is special precovering if and only if the ring $R$ is left noetherian and such that the character modules of left Gorenstein injective modules are Gorenstein flat right $R$-modules.

We also prove that the class of Gorenstein injective left $R$-modules is covering if and only if it is closed under direct limits. This result supports Enochs' conjecture. Enochs proved that a precovering class of modules $\mathcal{C}$ which is also closed under direct limits, is, in fact, a covering class (\cite{enochs:00:relative}, Corollary 5.2.7).
He also conjectured that "Every covering class of modules is closed under direct limits".
The conjecture has been verified for various type of classes. We now add the class of Gorenstein injective modules, $\mathcal{GI}$, to the list of classes of modules satisfying Enochs' conjecture.


We start by showing (Proposition 2) that the class of Gorenstein injective left $R$-modules, $\mathcal{GI}$, being closed under direct limits implies that it is a covering class. In \cite[Proposition 2]{iacob2023},  we proved that if $\mathcal{GI}$ is closed under direct limits then the ring $R$ is left noetherian and such that character modules of Gorenstein injectives are Gorenstein flat.
We also proved (\cite[Lemma 2]{iacob2023}) that, over such rings, $\mathcal{GI}$ is the left half of a duality pair. Therefore, (by \cite{holm:10:duality}, Theorem 3.1), $\mathcal{GI}$ being closed under direct limits implies that $\mathcal{GI}$ is a covering class. 


Then we prove (Theorem 2) that if the class of Gorenstein injective modules is special precovering, then it is closed under direct limits. In particular, this is the case when $\mathcal{GI}$ is covering.
The proof uses Proposition 3: "Let $\mathcal{W}$ be a class of left $R$-modules that is closed under direct summands, under taking cokernels of pure monomorphisms, and under
pure transfinite extensions. If $\mathcal{W}$ is closed under drect sums then $\mathcal{W}$ is closed under direct limits". Proposition 4 and Lemma 4 verify that, if $\mathcal{GI}$ is special precovering, then all these hypotheses are met.\\

We obtain a characterization of the rings over which $\mathcal{GI}$ is special precovering (Theorem 3):\\
The following statements are equivalent:\\
(1) The class of Gorenstein injective modules, $\mathcal{GI}$ is covering\\
(2) The class of Gorenstein injective modules, $\mathcal{GI}$ is special precovering.\\
(3) The class of Gorenstein injective modules is closed under direct limits.\\
(4) The ring $R$ is left noetherian and such that the character modules of Gorenstein injectives are Gorenstein flat.

Theorem 4 extends the result to complexes by showing that the class of Gorenstein injective complexes is special precovering if and only if it is covering if and only if it is closed under direct limits.



\section{preliminaries}
Throughout the paper, $R$ denotes an associative ring with unity. Unless otherwise specified, by module we mean a left $R$-module. $R-Mod$ denotes the category of left $R$-modules.

We recall the definition of Gorenstein injective modules. We will use $\mathcal{GI}$ to denote this class of modules.

\begin{definition} (\cite{enochs:95:gorenstein})
A module $M$ is called Gorenstein injective if there is an exact complex of injective modules $$\textbf{E} = \ldots \rightarrow E_1 \rightarrow E_0 \rightarrow E_{-1} \rightarrow \ldots$$ such that $M = Z_0\textbf{E}$, and such that the complex $Hom(I, \textbf{E})$ is exact for any injective module $I$.
\end{definition}

Since we use Gorenstein flat modules as well, we recall that they are the cycles of the exact complexes of flat modules that remain exact when tensored with any injective module. We use $\mathcal{GF}$ to denote this class of modules.

We also recall the definitions for Gorenstein injective precovers, covers, and special precovers. \\
\begin{definition} A homomorphism $\phi: G \rightarrow M$ is a Gorenstein injective precover of $M$ if $G$ is Gorenstein injective and if for any Gorenstein injective module $G'$ and any $\phi' \in Hom(G', M)$ there exists $u \in Hom(G', G)$ such that $\phi' = \phi u$. \\
A Gorenstein injective precover $\phi$ is said  to be a cover if any $v \in End_R(G)$ such that $\phi v = \phi$ is an automorphism of $G$.\\
A Gorenstein injective precover $\phi$ is said  to be special if $ker$ $\phi$ is in the right orthogonal class of that of Gorenstein injective modules, $\mathcal{GI} ^\bot$ (where $\mathcal{GI} ^\bot = \{ M | Ext^1(G, M)=0$, for all Gorenstein injective modules $G$ $\}$).
\end{definition}

As mentioned, the importance of the Gorenstein injective (pre)covers comes from the fact that they allow defining the Gorenstein injective resolutions: if the ring $R$ is such that every $R$-module $M$ has a Gorenstein injective precover then for every $M$ there exists a $Hom(\mathcal{GI}, -)$ exact complex $\ldots \rightarrow G_1 \rightarrow G_0 \rightarrow M \rightarrow 0$ with all $G_i$ Gorenstein injective modules. This is equivalent to $G_0 \rightarrow M$, and each $G_i \rightarrow Ker( G_{i-1} \rightarrow G_{i-2})$ being Gorenstein injective precovers. Such a complex is called a Gorenstein injective resolution of $M$; it is unique up to homotopy so it can be used to compute right derived functors of $Hom$.\\

If $\mathcal{GI}$ is a covering class, then working with a $\mathcal{GI}$-cover at every step, one can construct a minimal Gorenstein injective resolution of $M$ (such a minimal resolution is unique up to an isomorphism).

We will also use duality pairs, so we recall their definition.
\begin{definition} (\cite{holm:10:duality})
A \emph{duality pair} over $R$ is a pair $(\mathcal{M},\mathcal{C})$, where $\mathcal{M}$ is a class of left $R$-modules and $\mathcal{C}$ is a class of right $R$-modules, satisfying the following conditions:
\begin{enumerate}
\item $M \in \mathcal{M}$ if and only if $M^+ \in \mathcal{C}$ (where $M^+$ is the character module of $M$, $M^+ = Hom_Z (M, Q/Z)$).
\item $\mathcal{C}$ is closed under direct summands and finite direct sums.
\end{enumerate}
\end{definition}

A duality pair $(\mathcal{M},\mathcal{C})$ is called (co)product closed if the class $\mathcal{M}$ is
closed under (co)products in the category $R-Mod$.

\begin{theorem}\cite[Theorem 3.1]{holm:10:duality}\label{them-duality pair purity}
Let $(\mathcal{M},\mathcal{C})$ be a duality pair. Then the following hold:
\begin{enumerate}
\item $\mathcal{M}$ is closed under pure submodules, pure quotients, and pure extensions.
\item If $(\mathcal{M},\mathcal{C})$ is coproduct-closed then M is covering.
\end{enumerate}
\end{theorem}

\section{results}

We start by showing that $\mathcal{GI}$ being closed under direct limits implies that it is a covering class. The result follows from \cite[Theorem 3.1]{holm:10:duality}, \cite[Lemma 2]{iacob2023} and \cite[Proposition 2]{iacob2023}.

\begin{proposition} (\cite[Proposition 2]{iacob2023})
 If the class of Gorenstein injective left $R$-modules
is closed under direct limits, then the ring $R$ is left noetherian and
the character module of every Gorenstein injective left $R$-module is a
Gorenstein flat right $R$-module.
\end{proposition}

\begin{lemma} (\cite[Lemma 2]{iacob2023})
Let $R$ be a left noetherian
ring such that the character module of every Gorenstein injective left
$R$-module is a Gorenstein flat right R-module. Then $(\mathcal{GI},\mathcal{GF})$ is a duality pair.
\end{lemma}

\begin{proposition}
If the class of Gorenstein injective left $R$-modules is closed under direct limits then $\mathcal{GI}$ is a covering class.
\end{proposition}

\begin{proof}
By \cite[Lemma 2]{iacob2023} and \cite[Proposition 2]{iacob2023}, $\mathcal{GI}$ is the left half of a duality pair. Since $\mathcal{GI}$ is also closed under direct sums, it follows that it is a covering class (by \cite[Theorem 3.1]{holm:10:duality}).
\end{proof}


We prove that if the class of Gorenstein injective modules is special precovering then it is closed under direct limits. In particular, this is the case when $\mathcal{GI}$ is covering. We use the following result.\\

\begin{lemma} (\cite[Lemma 2.8]{bazzoni2024})
Let $\kappa$ be an infinite regular cardinal and $\mathcal{N} = (N_{\alpha},g_{{\alpha}{\beta}}: N_{\beta}\rightarrow N_{\alpha} , \beta < \alpha < \kappa)$ be a well-ordered $\kappa$-continuous directed system of modules. Assume that $N_0 = 0$. Put $C = \underrightarrow{lim} N$ and $N = \oplus_{\alpha < \kappa} N_{\alpha}$. Then the canonical ($\kappa$-pure) presentation of $C$
$$ \epsilon: 0 \rightarrow K \xrightarrow{\subseteq} N \xrightarrow{f} C \rightarrow 0$$
 is the direct limit of the $\kappa$-continuous well-ordered directed system of pure short exact sequences
 $$\varepsilon_{\alpha}: 0 \rightarrow K_{\alpha} \xrightarrow{\subseteq} \oplus_{\beta < \alpha} N_{\beta} \xrightarrow{f_{\alpha}} N_{\alpha '}\rightarrow 0, \alpha < \kappa$$
where we set $\alpha' = \alpha$ if $\alpha$ is limit (including zero), and $\alpha' = \alpha - 1$ if $\alpha$ is a successor ordinal. Furthermore, $f_{\alpha} | N_{\beta} = g_{{\alpha'}{\beta}}$ for each $\beta < \alpha$. Finally, the connecting morphisms between $\varepsilon_{\beta}$ and $\varepsilon_{\alpha}$, for $\beta < \alpha$, comprise (from left to right) of inclusion, canonical split inclusion and $g_{{\alpha '}{\beta '}}$. (We put $g_{{\gamma}{\gamma}} = id_{N_{\gamma}}$ for each $\gamma < \kappa$.)\\
 In particular, setting $K_{\kappa} = K$, we get a pure filtration $\mathfrak{K }= (K_{\alpha} , \alpha \le \kappa)$ of $K $.
\end{lemma}

\begin{proposition}
Let $\mathcal{W}$ be a class of modules that is closed under direct summands, under taking cokernels of pure monomorphisms, and under pure transfinite extensions. If $\mathcal{W}$ is closed under direct sums then $\mathcal{W}$ is closed under direct limits.
\end{proposition}

\begin{proof}
Let $\kappa$ be an infinite regular cardinal and $\mathcal{N} = (N_{\alpha}, g_{{\alpha}{\beta}}: N_{\beta}\rightarrow N_{\alpha} , \beta < \alpha < \kappa)$ bea well-ordered) $\kappa$-continuous directed system of modules, with $N_0 = 0$. There is a pure exact sequence $0 \rightarrow K \xrightarrow{\subseteq} N \xrightarrow{f} C \rightarrow 0$, where  $C = \underrightarrow{lim} N$ and $N = \oplus_{\alpha < \kappa} N_{\alpha}$. .\\
As  in Lemma 2, there is a pure filtration of $K = Ker (f)$, $\mathfrak{K }= (K_{\alpha} , \alpha \le \kappa)$ . Since the right connecting map between the short exact sequences $\varepsilon_{\alpha}$ and $\varepsilon_{{\alpha}+1}$ is $Id_{N_{\alpha}}$ (when $\alpha$ is limit), we have that $K_{{\alpha}+1}/K_{\alpha} \simeq N_{\alpha} \in \mathcal{W}$. \\If $\alpha$ is a successor then the exact sequences
 $$\varepsilon_{\alpha}: 0 \rightarrow K_{\alpha} \xrightarrow{\subseteq} \oplus_{\beta < \alpha} N_{\beta} \xrightarrow{f_{\alpha}} N_{\alpha -1}\rightarrow 0$$ is split exact. Therefore $K_{\alpha}$ is a direct summand of $\oplus_{\beta < \alpha} N_{\beta} \in \mathcal{W}$, so $K_{\alpha} \in \mathcal{W}$. The pure exact sequence $0 \rightarrow K_{\alpha} \rightarrow K_{\alpha +1} \rightarrow K_{{\alpha}+1}/K_{\alpha} \rightarrow 0$ with both $K_{\alpha}$ and $K_{\alpha +1}$ in $\mathcal{W}$ gives that $K_{{\alpha}+1}/K_{\alpha}$ is also in $\mathcal{W}$.\\
So $(K_{\alpha})_{\alpha}$ is a pure $\mathcal{W}$-filtration, and therefore $K \in \mathcal{W}$. \\
Since $N$ is also in $\mathcal{W}$, and $\mathcal{W}$ is closed under cokernels of pure monomorphisms it follows that  $C \in \mathcal{W}$.
\end{proof}

Our main application of Proposition 4 is proving that the class of Gorenstein injective modules, $\mathcal{GI}$, is special precovering if and only if it is closed under direct limits.\\

First, some quick remarks:\\

\begin{lemma}
If $\mathcal{GI}$ is special precovering, then $R$ is a left noetherian ring.
\end{lemma}

\begin{proof}
Let $(A_i)_{i \in I}$ be a family of injective left $R$-modules. By hypothesis, $\oplus_{i \in I} A_i \in \mathcal{GI}$. Since for each $i$, $A_i \in ^\bot \mathcal{GI}$ and $^\bot \mathcal{GI}$ is closed under direct sums (as the left half of a cotorsion pair), we have that $\oplus_{i \in I} A_i \in ^\bot \mathcal{GI}$. So $\oplus_{i \in I} A_i \in ^\bot \mathcal{GI} \bigcap \mathcal{GI} = Inj$. Since the class of injective left $R$-modules is closed under direct sums, $R$ is left noetherian.
\end{proof}

\begin{lemma}
The class of Gorenstein injective modules is closed under direct summands and under cokernels of monomorphisms.
\end{lemma}

\begin{proof}
Since $\mathcal{GI}$ is the right half of a hereditary cotorsion pair (by \cite{saroch.stovicek}) it follows that $\mathcal{GI}$ is closed under cokernels of monomorphisms. Also, as the right half of a cotorsion pair, $\mathcal{GI}$ is closed under direct summands.
\end{proof}

 In order to prove that $\mathcal{GI}$ being a specal precoverng covering class implies that it is closed under direct limits, we will also use the following result (this is basically \cite[Proposition 2]{EEI}, with a new proof)\\

\begin{proposition} (\cite[Proposition 2]{EEI})
If every $R$-module has a special Gorenstein injective precover then the class of Gorenstein injective modules is closed under transfinite extensions.
\end{proposition}

\begin{proof}

Let $(G_\alpha, \alpha \le \lambda)$ be a direct system of monomorphisms, with each $G_\alpha \in \mathcal{GI}$, and let $G= \underrightarrow{lim} G_\alpha$. Since, for each $\alpha$, we have that $G_{\alpha} \in ^\bot (\mathcal{GI} ^\bot)$, it follows that $G= \underrightarrow{lim} G_\alpha \in ^\bot (\mathcal{GI} ^\bot)$ by Eklof Lemma (\cite[Theorem 1.2]{eklof}).

The class $\mathcal{GI}$ is precovering and closed under direct summands, so $\mathcal{GI}$ is closed under direct sums (\cite[Lemma 9.14]{trlifaj}). Thus $\oplus _{\alpha \le \lambda} G_{\alpha} \in \mathcal{GI}$.\\
Since there is a short exact sequence $0 \rightarrow K \rightarrow \oplus _{\alpha \le \lambda} G_{\alpha} \rightarrow G \rightarrow 0$ with $\oplus _{\alpha \le \lambda} G_{\alpha}$ Gorenstein injective, it follows that any Gorenstein injective precover of $G$ has to be surjective.\\

The class $\mathcal{GI}$ is special precovering, so there is an exact sequence $0 \rightarrow A \rightarrow \overline{G} \rightarrow G \rightarrow 0$ with $A \in \mathcal{GI}^\bot$ and $\overline{G}$ Gorenstein injective. But $G \in ^\bot (\mathcal{GI} ^\bot)$, so we have that $Ext^1(G, A)=0$. Thus $G$ is a direct summand of $\overline{G}$, therefore $G$ is Gorenstein injective.
\end{proof}

\begin{corollary}
If the class of Gorenstein injective modules is covering, then it closed under transfinite extensions.
\end{corollary}

\begin{proof}
 By \cite[Corollary 7.2.3]{enochs:00:relative}, any Gorenstein injective cover is a special precover, so the result follows from Proposition 4.
\end{proof}



\begin{theorem}
The class of Gorenstein injective modules is special precovering if and only if it is closed under direct limits.
\end{theorem}

\begin{proof}
  By Proposition 2, if $\mathcal{GI}$ is closed under direct limits, then it is a covering class, hence special precovering.\\
 Conversely, assume that $\mathcal{GI}$ is special precovering. Then, by Proposition 4, $\mathcal{GI}$ is closed under transfinite extensions. \\
It is known that a precovering class of modules that is closed under direct summands is also closed under direct sums. Since $\mathcal{GI}$ is closed under direct summands, direct sums, cokernels of monomorphisms, and transfinite extensions, it follows (by Proposition 3) that $\mathcal{GI}$ is closed under direct limits.
\end{proof}

In \cite{iacob2023} we gave a characterization of the rings for which the class of Gorenstein injective modules is closed under direct limits. Using \cite[Theorem 2]{iacob2023}, and Theorem 2 above we obtain:\\

\begin{theorem}
The following statements are equivalent:\\

(1) The class of Gorenstein injective modules, $\mathcal{GI}$ is covering\\
(2) The class of Gorenstein injective modules, $\mathcal{GI}$ is special precovering.\\
(3) The class of Gorenstein injective modules is closed under direct limits.\\
(4) The ring $R$ is left noetherian and such that the character modules of Gorenstein injectives are Gorenstein flat.
\end{theorem}

\begin{proof}
(1) $\Rightarrow$ (2) is immediate since a covering class is special precovering (\cite[Corollary 7.2.3]{enochs:00:relative}).\\
(2) $\Rightarrow$ (3) by Theorem 2 above.\\
(3) $\Rightarrow$ (1) by Proposition 2. \\
(3) $\Leftrightarrow$ (4) by \cite[Theorem 2]{iacob2023}.
\end{proof}

The following theorem shows that the result extends to the category of complexes of left $R$-modules $Ch(R)$.

\begin{theorem}
The following are equivalent:\\
(1) The class of Gorenstein injective left $R$-modules, $\mathcal{GI}$, is closed under direct limits.\\
(2) The class of Gorenstein injective left $R$-modules is covering in $R$-Mod.\\
(3) The class of Gorenstein injective complexes is closed under direct limits.\\
(4) The class of Gorenstein injective complexes is covering in $Ch(R)$.\\
(5) The class of Gorenstein injective complexes is special precovering.
\end{theorem}

\begin{proof}
(1) $\Leftrightarrow$ (2) by Theorem 3.\\
(1) $\Rightarrow$ (3) Since $\mathcal{GI}$ is covering in this case, $R$ is a left noetherian ring (Lemma 3). By \cite[Theorem 8]{liu}, a complex is Gorenstein injective if and only if it is a complex of Gorenstein injective modules. Since $\mathcal{GI}$ is closed under direct limits, it follows that the class of Gorenstein injective complexes is also closed under direct limits.\\
(3) $\Rightarrow$ (1) Let $(G_i)_{i \in I}$ be a family of Gorenstein injective modules. For each $i \in I$ consider the complex $S^0(G_i) = 0 \rightarrow G_i \rightarrow 0$ with $G_i$ in the zeroth place. By hypothesis, $\underrightarrow{lim} S^0(G_i)$ is a Gorenstein injective complex. By \cite[Theorem 8]{liu}, $\underrightarrow{lim} {G_i} \in \mathcal{GI}$.\\
(2) $\Rightarrow$ (4) By \cite[Theorem 1]{I}, if $\mathcal{C}$ is a covering class of modules such that $\mathcal{C}$ is closed under direct limits and extensions, then the class of complexes of modules from $\mathcal{C}$, $dw \mathcal{C}$, is covering in $Ch(R)$. So if $\mathcal{GI}$ is covering (hence closed under direct limits) then the class of Gorenstein injective complexes, $dw \mathcal{GI}$, is covering in $Ch(R)$.\\
(4) $\Rightarrow$ (5) is immediate, since a covering class is special precovering.\\
(5) $\Rightarrow$ (2) Let $M$ be a left $R$-module and consider the complex $ S^0(M) =  0 \rightarrow M \rightarrow 0$ with $M$ in the zeroth place. Let $G \rightarrow S^0(M)$ be a special Gorensten injective precover with kernel $A$.\\

Let $G' \in \mathcal{GI}$. We have that $D^n(G') = 0 \rightarrow G' = G' \rightarrow 0$ with $G'$ in places $n$ and $n+1$ is a Gorenstein injective complex, so $Ext^1 (D^n(G'), A)= 0$. But $Ext^1(D^n(G'), A) \simeq Ext^1_R(G', A_n)$, for each $n$. So $A_n \in  \mathcal{GI} ^\bot$, for each $n$. In particular $A_0 \in \mathcal{GI} ^\bot$.\\
By hypothesis, the sequence $$0 \rightarrow Hom_{Ch(R)} (D^n(G'), A) \rightarrow Hom_{Ch(R)} (D^n(G'), G)
 \rightarrow Hom_{Ch(R)} (D^n(G'), S^0(M)) \rightarrow 0$$ is exact. But $Hom_{Ch(R)} (D^n(G'), A) \simeq Hom(G', A_0)$, $Hom_{Ch(R)} (D^n(G'), G) \simeq Hom(G', G_0)$, and $Hom_{Ch(R)} (D^n(G'), S^0(M)) \simeq Hom(G', M)$. So $0 \rightarrow Hom(G', A_0) \rightarrow Hom(G', G_0) \rightarrow Hom(G', M) \rightarrow 0$ is exact. Thus $G_0 \rightarrow M$ is a special Gorenstein injective precover.\\
  
  Since $\mathcal{GI}$ is special precovering, it is a covering class (by Theorem 3).
\end{proof}


\end{document}